\documentclass[reqno, 10pt]{amsart}
\usepackage{amssymb}
\usepackage{amsmath}
\usepackage[mathscr]{euscript}
\usepackage{hyperref}
\usepackage{graphicx}
\usepackage[normalem]{ulem}

\usepackage{amssymb}
\usepackage{hyperref}
\RequirePackage[dvipsnames]{xcolor} 
\definecolor{halfgray}{gray}{0.55} 
\definecolor{webgreen}{rgb}{0,0.5,0}
\definecolor{webbrown}{rgb}{.6,0,0} \hypersetup{%
  colorlinks=true, linktocpage=true, pdfstartpage=3,
  pdfstartview=FitV,%
  breaklinks=true, pdfpagemode=UseNone, pageanchor=true,
  pdfpagemode=UseOutlines,%
  plainpages=false, bookmarksnumbered, bookmarksopen=true,
  bookmarksopenlevel=1,%
  hypertexnames=true,
  pdfhighlight=/O,
  urlcolor=webbrown, linkcolor=RoyalBlue,
  citecolor=webgreen, 
  pdftitle={},%
  pdfauthor={},%
  pdfsubject={2000 MAthematical Subject Classification: Primary:},%
  pdfkeywords={},%
  pdfcreator={pdfLaTeX},%
  pdfproducer={LaTeX with hyperref}%
}

\usepackage{enumitem}

\newtheorem{theorem}{Theorem}[section]
\newtheorem{lemma}[theorem]{Lemma}
\newtheorem{corollary}[theorem]{Corollary}
\newtheorem{proposition}[theorem]{Proposition}

\newtheorem{example}[theorem]{Example}

\def\R{\mathbb{R}}
\def\N{\mathbb{N}}
\def\Reg{\mathcal{R}^{\mu}}
\newcommand{\norm}[1]{{\left\lVert \, #1 \, \right\rVert}}

\begin{document}

\title
[On the periodic approximation of Lyapunov exponents]
{On the periodic approximation of Lyapunov exponents for semi-invertible cocycles}

\author{Lucas Backes}

\address{\noindent Departamento de Matem\'atica, Universidade Federal do Rio Grande do Sul, Av. Bento Gon\c{c}alves 9500, CEP 91509-900, Porto Alegre, RS, Brazil.
\newline e-mail: \rm
  \texttt{lhbackes@impa.br} }

\date{\today}

\keywords{Semi-invertible linear cocycles, Lyapunov exponents, periodic points, approximation}
\subjclass[2010]{Primary: 37H15, 37A20; Secondary: 37D25}

\begin{abstract}
We prove that, for semi-invertible linear cocycles, Lyapunov exponents of ergodic measures may be approximated by Lyapunov exponents on periodic points.
\end{abstract}

\maketitle

\section{Introduction}

A very general and also vague idea that appears in the study of dynamical systems is that ``if a system exhibits enough hyperbolicity then most of its dynamical interesting information is concentrated in its periodic orbits''. There are many examples supporting this idea. For instance, it is known that cohomology classes of H\"older cocycles over hyperbolic systems are characterized by its information on periodic points (see for instance \cite{Liv71, Liv72, Kal11, dLW10, Bac15, Sa15, BK16, KP16} and references therein), equilibrium states associated to different potentials coincide whenever those potentials have the same information on periodic points \cite{Bow75} and so on.

In the present work we also present an example supporting the previous ``belief'' in the context of Lyapunov exponents of \emph{semi-invertible} linear cocycles. More precisely, we show that (see Section \ref{sec: statements} for precise definitions and statements)
\begin{theorem}\label{theo: introduction}
Lyapunov exponents of ergodic measures may be approximated by Lyapunov exponents on periodic points.
\end{theorem}
In other words, all the information carried by the Lyapunov exponents of $(f,A)$ is indeed concentrated on periodic points.

The objects involved in our example are very classical in the fields of Dynamical Systems and Ergodic Theory and can be defined as follows: given an invertible ergodic measure preserving dynamical system $f:M \rightarrow M$ defined on a measure space $(M,\mathcal{A},\mu)$ and a measurable matrix-valued map $A:M\rightarrow M(d, \mathbb{R})$, the pair $(f,A)$ is called a \textit{semi-invertible linear cocycle} (or just \textit{linear cocycle} for short). Sometimes one calls linear cocycle (over $f$ generated by $A$), instead, the sequence $\lbrace A^n\rbrace _{n\in \mathbb{N}}$ defined by
\begin{equation}\label{def:cocycles}
A^n(x)=
\left\{
	\begin{array}{ll}
		A(f^{n-1}(x))\ldots A(f(x))A(x)  & \mbox{if } n>0 \\
		Id & \mbox{if } n=0 \\
	\end{array}
\right.
\end{equation}
for all $x\in M$. The word `semi-invertible' refers to the fact that the action of the underlying dynamical system $f$ is invertible while the action on the fibers given by $A$ may fail to be invertible. We refer to the Introduction of \cite{DrF} for some interesting applications of semi-invertible cocycles.

Under certain integrability conditions, it was proved in \cite{FLQ10} that there exists a full $\mu$-measure set $\Reg \subset M$, whose points are called $\mu$-regular points, such that for every $x\in \Reg$ there exist numbers $\lambda _1>\ldots > \lambda _{l}\geq -\infty$, called \textit{Lyapunov exponents}, and a direct sum decomposition $\mathbb{R}^d=E^{1,A}_{x}\oplus \ldots \oplus E^{l,A}_{x}$ into vector subspaces which are called \textit{Oseledets subspaces} and depend measurable on $x$ such that, for every $1\leq i \leq l$,
\begin{itemize}
\item dim$(E^{i,A}_{x})$ is constant,
\item $A(x)E^{i,A}_{x}\subseteq E^{i,A}_{f(x)}$ with equality when $\lambda _i>-\infty$
\end{itemize}
and
\begin{itemize}
\item $\lambda _i =\lim _{n\rightarrow +\infty} \dfrac{1}{n}\log \parallel A^n(x)v\parallel$ 
for every non-zero $v\in E^{i,A}_{x}$.
\end{itemize}
This result extends a famous theorem due to Oseledets \cite{Ose68} known as the \textit{multiplicative ergodic theorem} which was originally stated in both, \textit{invertible} (both $f$ and the matrices are assumed to be invertible) and \textit{non-invertible} (neither $f$ nor the matrices are assumed to be invertible) settings (see also \cite{LLE}). While in the invertible case the conclusion is similar to the conclusion above (except that all Lyapunov exponents are finite), in the non-invertible case, instead of a direct sum decomposition into invariant vector subspaces, one only get an invariant \textit{filtration} (a sequence of nested subspaces) of $\R ^d$.

In the invertible setting, Theorem \ref{theo: introduction} was already gotten by Kalinin in \cite{Kal11} extending a theorem of Wang and Sun \cite{WS10} on the approximation of Lyapunov exponents of hyperbolic invariant measures for diffeomorphims. In fact, the proof of our main result is based on ideas from those works. The lack of invertibility of the matrices, however, brings in some additional difficulties. To deal with it, we introduce the notion of \emph{Lyapunov norm} for semi-invertible cocycles and present some useful properties about these objects. 

It is also worth noticing that a similar approximation result was gotten by Dai \cite{Dai10} in the case when just the matrices are assumed to be invertible. More recently, Kalinin and Sadovskaya \cite{KS} addressed a similar problem in the invertible setting but when the cocycle takes values in the set of invertible operators of a Banach space. In such setting, Theorem \ref{theo: introduction} can not be fully recovered.

\section{Statements}\label{sec: statements}

Let $(M,d)$ be a compact metric space, $\mu$ a measure defined on the Borel sets of $(M,d)$ and $f: M \to M $ a measure preserving homeomorphism. Assume also that $\mu$ is ergodic.

We say that $f$ satisfies the \textit{Anosov Closing property} if there exist $C_1 ,\varepsilon _0 ,\theta >0$ such that if $z\in M$ satisfies $d(f^n(z),z)<\varepsilon _0$ then there exists a periodic point $p\in M$ such that $f^n(p)=p$ and
\begin{displaymath}
d(f^j(z),f^j(p))\leq C_1 e^{-\theta \min\lbrace j, n-j\rbrace}d(f^n(z),z)
\end{displaymath}
for every $j=0,1,\ldots ,n$. Notice that shifts of finite type, basic pieces of Axiom A diffeomorphisms and more generally, hyperbolic homeomorphisms are particular examples of maps satisfying the Anosov Cloisng property. See for instance, \cite{KH95} p.269, Corollary 6.4.17.

Given a continuous map $A: M\to M(d,\mathbb{R})$ such that $\int \log ^{+}\norm{A(x)}d\mu (x)<\infty$, let us denote by 
\begin{displaymath}
\lambda _1 (A,\mu)>\lambda _2 (A,\mu) > \dots >\lambda _l (A,\mu)\geq -\infty
\end{displaymath}
the Lyapunov exponents of the cocycle $(f,A)$ with respect to the measure $\mu$ and by
$$\gamma _1(A,\mu)\geq \gamma _2(A,\mu)\geq \ldots \geq \gamma _d(A,\mu)$$
the Lyapunov exponents of $(f,A)$ with respect to $\mu$ counted with multiplicities. Given a periodic point $p$, we denote its Lyapunov exponents and Lyapunov exponents counted with multiplicities by $\{\lambda _i(A,p)\}_{i=1}^l$ and $\{\gamma _i(A,p)\}_{i=1}^d$, respectively. When there is no risk of ambiguity, we suppress the index $A$ or even both $A$ and $\mu$ from the previous objects. 

In what follows we are also going to assume that $A:M\to M(d,\R)$ is an $\alpha$-H\"{o}lder continuous map. This means that there exists a constant $C_2>0$ such that
\begin{displaymath}
\norm{A(x)-A(y)} \leq C_2 d(x,y)^{\alpha}
\end{displaymath}
for all $x,y\in M$ where $\norm{A}$ denotes the operator norm of a matrix $A$, that is, $\norm{A} =\sup \lbrace \norm{Av}/\norm{v};\; \norm{v}\neq 0 \rbrace$.

\subsection{Main results}\label{subsec: main results} The main result of this work is the following one
\begin{theorem}\label{theo: main} 
Let $f: M\to M $ be a homeomorphism satisfying the Anosov Closing property, $\mu$ an ergodic $f$-invariant probability measure and $A:M\to M(d,\R)$ an $\alpha$-H\"{o}lder continuous map. Then, there exists a sequence of periodic points $(p_k)_{k\in \mathbb{N}}$ such that
\begin{displaymath}
	\gamma _i (A,p_k)\xrightarrow{k\to +\infty} \gamma _i (A,\mu)
\end{displaymath} 
for every $i=1,\ldots ,d$.
\end{theorem}

As a simple consequence of our main result we get the following corollary. In order to state it, let us assume that $f$ and $A$ satisfy the hypotheses of Theorem \ref{theo: main}. Then,

\begin{corollary}\label{cor: bounded}
If all Lyapunov exponents of $(f,A)$ on periodic points are uniformly bounded by bellow then $A(x)\in GL(d, \R)$ for almost every $x\in M$ with respect to \emph{every} $f$-invariant probability measure $\mu$. 
\end{corollary}

\begin{proof}
If there exist a measure $\mu$, which by the the Ergodic Decomposition Theorem may be assumed to be ergodic, and a set $B\subset M$ with positive $\mu$-measure such that $A(x)\notin GL(d, \R)$ for every $x\in B$ then $\gamma _d(A,\mu)=-\infty$ which in light of Theorem \ref{theo: main} contradicts our assumption.
\end{proof}

We observe that satisfying $A(x)\in GL(d, \R)$ for almost every $x\in M$ with respect to \emph{every} $f$-invariant probability measure as in the previous corollary does not imply, in general, that $A(x)\in GL(d, \R)$ for every $x\in M$. Indeed,

\begin{example}
\normalfont Let $M=\{0,1\}^{\mathbb{Z}}$ be the space of bilateral sequences in zeros and ones and $f :M\to M$ be the left shift $f((x_i)_{i\in \mathbb{Z}})=(x_{i+1})_{i\in \mathbb{Z}}$. Given $\theta \in (0,1)$ we endow $M$ with the distance
\begin{displaymath}
d(x,y)=\theta ^{N(x,y)},\; \textrm{where} \; N(x,y)=\max \lbrace N\geq 0; x_n=y_n \; \textrm{for all} \mid n\mid < N \rbrace
\end{displaymath}
for $x=(x_i)_{i\in \mathbb{Z}}$ and $y=(y_i)_{i\in \mathbb{Z}}$. A very well known fact is that $(M,d)$ is a compact metric space and $f$ is a homeomorphism satisfying the Anosov Closing property. Let $q=(q_i)_{i\in \mathbb{Z}}\in M$ be such that $q_i=1$ for every $i\neq -1$ and $q_{-1}=0$ and fix $a>1$ so that $a\theta >1$. Consider $A:M\to \mathbb{R}$ given by 
\begin{displaymath}
	A(x)=\left\{\begin{array}{cc}
		\frac{a}{\theta ^3} d (x,q)& \mbox{if} \quad d (x,q)\leq \theta ^{3}\\
		a & \mbox{if}\quad d (x,q)> \theta ^{3}.
	\end{array}
	\right.
\end{displaymath}
Then, $A$ is a H\"older map and generates a semi-invertible cocycle over $f$. Moreover, $A(q)=0$ and $(f,A)$ has \emph{all} Lyapunov exponents on periodic points uniformly bounded by bellow by $\log (a\theta)> 0$.  
In fact, if a periodic point $p\in M$ is such that $d(f^j(p),q)>\theta ^3$ for every $j\in \mathbb{N}$ then obviously $\lambda _1(p)=\log a>0$. Now, suppose there exists $j\in \mathbb{N}$ so that $d(f^j(p),q)\leq \theta ^3$. We may assume without loss of generality that $j=0$. More precisely, suppose $p=(p_i)_{i\in \mathbb{Z}}$ satisfies $p_i=q_i$ for every $\mid i\mid\leq n$ and $p_i\neq q_i$ for some $i$ so that $\mid i\mid =n+1$ with $n\geq 2$. Then, $A(p)=\frac{a}{\theta ^3}\theta ^{n+1}$ and $A(f^j(p))=a$ for, at least, $j=1, 2,\ldots ,n+1$. In particular, $A^{n+2}(p)=a^{n+2}\theta ^{n-2}$ and $\frac{1}{n+2}\log \norm{A^{n+2}(p)}\geq \log (a\theta)> 0$. In other words, whenever a periodic point $p$ is $\theta ^{n+1}$ close to $q$, its next $n+1$ iterations are going to be out of the ball of radius $\theta^3$ and centered at $q$. Repeating this argument we can construct a sequence $(n_k)_{k\in \mathbb{N}}$ going to $+\infty$ such that  
$$\frac{1}{n_k}\log \norm{A^{n_k}(p)} \geq \log (a\theta)> 0$$
for every $k\in \mathbb{N}$. In particular, $\lambda _1(p)\geq  \log (a\theta)>0$ as claimed. Consequently, by Corollary \ref{cor: bounded} $A(x)\in GL(1, \R)$ for almost every $x\in M$ with respect to every $f$-invariant probability measure and $A(q)\notin GL(1, \R)$. Observe that, besides the choice of $A$, the main feature underlying our construction is that $q\in W^s(p)\cap W^u(p)$ where $p$ is the fixed point $p=(p_i)_{i\in \mathbb{Z}}$ such that $p_i=1$ for every $i\in \mathbb{Z}$. Another simple remark is that such example can be constructed in any dimension. We take $A: M\to M(d,\mathbb{R})$ to be such that $A(x)$ is a diagonal matrix for every $x\in M$ where one of its entries is the map constructed above while the others are all constant. 
\end{example}

The previous example also reveals very different behaviors of invertible and semi-invertible cocycles. Indeed, it was proved by Cao in \cite{Cao03} that if an invertible cocycle $A$ defined over a homeomorphism $f$ has only positive Lyapunov exponents with respect to every $f$-invariant probability measure then it is uniformly expanding. The previous example shows us that this is no longer true in the semi-invertible context.

\section{Lyapunov norm} In order to estimate the growth of the cocycle $A$ along an orbit we introduce the notion of \textit{Lyapunov norm} for semi-invertible cocycles. This is based on a similar notion for \emph{invertible cocycles} (see for instance \cite{BP07}). 

Let $x\in \Reg$ be a regular point and $\mathbb{R}^d=E^{1,A}_{x}\oplus \ldots \oplus E^{l,A}_{x}$ be the Oseledets decomposition at $x$. Given $i\in \{1,\ldots , l-1 \}$ and $n\in \mathbb{N}$, let us consider the map
$$A^n(f^{-n}(x))_{\mid E^{i,A}_{f^{-n}(x)}}: E^{i,A}_{f^{-n}(x)}\to E^{i,A}_{x}$$
which is invertible and let us denote its inverse by $\left( A^n(f^{-n}(x))\right)^{-1}_i$. Now, for every $n\in \mathbb{Z}$ and $x\in \Reg$ let us consider the linear map $A^n_i(x):E^{i,A}_{x}\to E^{i,A}_{f^n(x)}$ given by
\begin{displaymath}
	A^n_i(x)u=\left\{\begin{array}{cc}
		A^n(x)_{\mid E^{i,A}_{x}}u & \mbox{if} \quad n\geq 0\\
		\left( A^{-n}(f^n(x))\right)^{-1}_i u & \mbox{if}\quad n<0.
	\end{array}
	\right.
\end{displaymath}
Observe that, for every $m,n\in \mathbb{Z}$,
\begin{equation}
A^{m+n}_i(x)=A^n_i(f^m(x))A^m_i(x).
\end{equation}
Indeed, for $m,n\geq 0$ it follows readily from the definition and \eqref{def:cocycles}. Suppose now $m,n>0$ and let us prove that $A^{-m-n}_i(x)=A^{-n}_i(f^{-m}(x))A^{-m}_i(x)$. We start observing that, since
$$\left( A^{m+n}(f^{-m-n}(x)) \right)_{\mid E^{i,A}_{f^{-m-n}(x)}}=\left( A^m(f^{-m}(x))A^{n} (f^{-m-n}(x))\right)_{\mid E^{i,A}_{f^{-m-n}(x)}} ,$$
it follows by the invariance of the Oseledets spaces that
$$\left( A^{m+n}(f^{-m-n}(x)) \right)_{\mid E^{i,A}_{f^{-m-n}(x)}}= A^m(f^{-m}(x))_{\mid E^{i,A}_{f^{-m}(x)}}A^{n} (f^{-m-n}(x))_{\mid E^{i,A}_{f^{-m-n}(x)}}.$$
Thus, taking the inverses on both sides the result follows. The case when $m$ and $n$ have different signs may be deduced by combining the previous two.

In order to define the \emph{Lyapunov norm} associated to the cocycle $A$ at a regular point $x\in \Reg$, we start by defining the \emph{Lyapunov inner product}: given $\delta >0$ and two vectors $u=u_1+\ldots +u_l $ and $v=v_1+\ldots +v_l $ in $\R ^d$ where $u_i, v_i \in E^{i,A}_x$ for every $1\leq i\leq l$, the \textit{$\delta$-Lyapunov inner product} of $u$ and $v$ is defined by
\begin{displaymath}
\langle u,v\rangle _{x,\delta}=\sum _{i=1}^l\langle u_i,v_i \rangle _{x,\delta,i}
\end{displaymath}
where 
\begin{equation}\label{eq: def Lyap i norm}
\langle u_i, v_i\rangle_{x,\delta,i}= \sum _{n\in \mathbb{Z}}\langle A^{n}_i(x)u_i, A^{n}_i(x)v_i \rangle e^{-2\lambda _i n -2\delta \mid n\mid}
\end{equation}
for every $i $ for which $\lambda _i$ is finite and 
\begin{equation}\label{eq: def Lyap l norm}
\langle u_l, v_l\rangle _{x,\delta,l}= \sum _{n=0}^{+\infty} \langle A^{n}(x)u_l,A^{n}(x)v_l \rangle e^{\frac{2}{\delta}n}
\end{equation}
in the case when $\lambda _l =-\infty$. Observe that both series \eqref{eq: def Lyap i norm} and \eqref{eq: def Lyap l norm} converge for any $x\in \Reg$. Indeed, convergence of the second one is easily to verify while the convergence of the first one follows from the next lemma whose proof is also going to be used in the sequel.

\begin{lemma}\label{lem: auxil 1}
For every $u\in E^{i,A}_x\setminus \{0\}$,
$$\lim _{n\to \pm \infty}\frac{1}{n}\log \norm{A^n_i(x)u}=\lambda _i.$$
\end{lemma}
\begin{proof}
The fact that $\lim _{n\to + \infty}\frac{1}{n}\log \norm{A^n_i(x)u}=\lambda _i$ follows from the definition and Oseledets' Theorem. Let us prove that $\lim _{n\to - \infty}\frac{1}{n}\log \norm{A^n_i(x)u}=\lambda _i$.
Given $\varepsilon >0$, it follows from Theorem 2 of \cite{DrF} that there exists a measurable map $C:M\to (0,\infty)$ such that
$$\frac{1}{C(f^{-n}(x))}e^{(\lambda _i-\varepsilon )n}\norm{v}\leq \norm{A^n(f^{-n}(x))v}\leq C(f^{-n}(x))e^{(\lambda _i+\varepsilon )n}\norm{v}$$
for every $n\in \mathbb{N}$ and $v\in E^{i,A}_{f^{-n}(x)}\setminus \{0\}$ and
$$C(f^n(x))\leq C(x)e^{\varepsilon\mid n\mid}$$
for any $n\in \mathbb{Z}$. Combining this two inequalities we get that
$$\frac{1}{C(x)}e^{(\lambda _i-2\varepsilon )n}\norm{v}\leq \norm{A^n(f^{-n}(x))v}\leq C(x)e^{(\lambda _i+2\varepsilon )n}\norm{v}$$
for every $n\in \mathbb{N}$ and $v\in E^{i,A}_{f^{-n}(x)}\setminus \{0\}$.

Let $u\in E^{i,A}_x\setminus \{0\}$ and $n\in \mathbb{N}$. Then, applying the previous inequality to $v=A^{-n}_i(x)u$ we get that
$$\frac{1}{C(x)}e^{(\lambda _i-2\varepsilon )n}\norm{A^{-n}_i(x)u}\leq \norm{u}\leq C(x)e^{(\lambda _i+2\varepsilon )n}\norm{A^{-n}_i(x)u}.$$
That is,
\begin{equation}\label{eq: comp norm lemma}
\frac{1}{C(x)}e^{-(\lambda _i+2\varepsilon )n}\norm{u}\leq \norm{A^{-n}_i(x)u}\leq C(x)e^{-(\lambda _i-2\varepsilon )n}\norm{u}.
\end{equation}
Taking the logarithm at each term, dividing by $-n$ and making $n\to +\infty$ it follows that
$$\lambda _i -2\varepsilon \leq \lim _{n\to +\infty}\frac{1}{-n}\log \norm{A^{-n}_i(x)u} \leq \lambda _i +2\varepsilon.$$
Now, since $\varepsilon >0$ is arbitrary the lemma follows.
\end{proof}

In particular, it follows that $\langle \cdot ,\cdot \rangle _{x,\delta}$ is actually an inner product in $\R^d$. We then define the \emph{$\delta$-Lyapunov norm} $\norm{.}_{x,\delta}$ associated to the cocycle $A$ at $x\in \Reg$ as the norm generated by $\langle \cdot ,\cdot \rangle _{x,\delta}$. When there is no risk of ambiguity, we just write $\norm{.}_x$ and $\norm{.}_{x,i}$ instead of $\norm{.}_{x,\delta}$ and $\norm{.}_{x,\delta,i}$ and call it just \emph{Lyapunov norm}.

Given a liner map $B\in M(d,\R)$, its Lyapunov norm is defined for any regular points $x,y\in \Reg$ by
\begin{displaymath}
\norm{B}_{y\leftarrow x}=\sup \{\norm{Bu}_{y}/\norm{u}_{x}; \; u\in \R^d\setminus \{0\} \}.
\end{displaymath}  

The next proposition gives us some useful properties of the Lyapunov norm that we are going to use in the sequel. In the case when the cocycle is invertible, similar properties of the Lyapunov norm are very well known. We now prove them in the semi-invertible setting.

\begin{proposition} Let $x\in \Reg$.

i) For every $1\leq i<l$ and $u\in E^{i,A}_x$ we have
\begin{equation} \label{ineq: Lyapunov norm}
e^{(\lambda _i -\delta)n}\norm{u}_{x,i}\leq \norm{A^n(x)u}_{f^n(x),i}\leq e^{(\lambda _i +\delta)n}\norm{u}_{x,i} 
\end{equation}
for every $n\in \N$;

ii) If $\lambda _l =-\infty$ then for every $u\in E^{l,A}_x$ and $n\in \mathbb{N}$ we have
$$\norm{A^n(x)u}_{f^n(x),l}\leq e^{-\frac{1}{\delta} n}\norm{u}_{x,l}; $$

iii) For $\delta >0$ such that $-\frac{1}{\delta}<\lambda _1$ we have 
\begin{equation}\label{ineq: norm}
 \norm{A^n(x)}_{f^n(x)\leftarrow x}\leq e^{(\lambda _1 +\delta)n} 
\end{equation}
for every $n\in \N$;

iv) For every $\delta >0$ such that $-\frac{1}{\delta}<\lambda _{l-1}$, there exists a measurable function $K_{\delta}:\Reg \to (0,+\infty)$ such that
\begin{equation}\label{ineq: norm x Lyapunov norm}
\norm{u}\leq \norm{u}_x\leq K_{\delta}(x)\norm{u}
\end{equation}
whose growth along any regular orbit is bounded; more precisely,
\begin{equation}\label{ineq: K delta growth}
K_{\delta}(x)e^{-\delta n}\leq K_{\delta}(f^n(x))\leq K_{\delta}(x)e^{\delta n} \quad \forall n\in \N.
\end{equation} 
Consequently, for any linear map $B$ and any regular points $x$ and $y$
\begin{equation}\label{ineq: norm x Lyapunov norm operator}
K_{\delta}(x)^{-1}\norm{B}\leq \norm{B}_{y\leftarrow x}\leq K_{\delta}(y)\norm{B}.
\end{equation}

\end{proposition}

\begin{proof}

In order to prove $i)$ we observe that for any $u\in E^{i,A}_x$,
\begin{displaymath}
\begin{split}
\norm{A(x)u}_{f(x),i}^2&= \sum _{n\in \mathbb{Z}}\norm{A^{n}_i(f(x))A(x)u}^2 e^{-2\lambda _i n -2\delta \mid n\mid}\\
&= \sum _{n\in \mathbb{Z}}\norm{A^{n+1}_i(x)u}^2 e^{-2\lambda _i n -2\delta \mid n\mid}\\
&= \sum _{n\in \mathbb{Z}}\norm{A^{n+1}_i(x)u}^2 e^{-2\lambda _i (n+1) -2\delta \mid n+1\mid} e^{2\lambda _i+ 2\delta (\mid n\mid -\mid n-1\mid )}.
\end{split}
\end{displaymath}
Consequently,
$$e^{(\lambda _i -\delta)}\norm{u}_{x,i}\leq \norm{A(x)u}_{f(x),i}\leq e^{(\lambda _i +\delta)}\norm{u}_{x,i} $$
which implies $i)$. Item $ii)$ is analogous. Indeed, we have that 
\begin{displaymath}
\begin{split}
\norm{A^n(x)u}^2_{f^n(x),l}&= \sum _{k=0}^{+\infty}\norm{A^{k}(f^n(x))A^n(x)u}^2 e^{\frac{2}{\delta} k}\\
&= \sum _{k=0}^{+\infty}\norm{A^{k+n}(x)u}^2 e^{\frac{2}{\delta}(k+n)}e^{-\frac{2}{\delta}n}\leq e^{-\frac{2}{\delta} n}\norm{u}^2_{x,l},\\
\end{split}
\end{displaymath}
for every $u\in E^{l,A}_x$.

In order to get $iii)$ one only have to observe that, for any $u\in \mathbb{R}^d$, 
\begin{displaymath}
\begin{split}
\norm{A^n(x)u}_{f^n(x)}^2&=\sum _{i=1}^l\norm{A^n(x)u_i}_{f^n(x),i}^2\\
&\leq   \sum _{i=1}^{l-1}e^{2(\lambda _i+\delta)n}\norm{u_i}_{x,i}^2 +e^{-2\frac{1}{\delta}n}\norm{u_l}_{x,l}^2\\
& \leq  e^{2(\lambda _1+\delta)n}  \sum _{i=1}^{l}\norm{u_i}_{x,i}^2=e^{2(\lambda _1+\delta)n} \norm{u}_{x}^2.
\end{split}
\end{displaymath}

The first inequality in $iv)$ is trivial. To prove the second one, we proceed analogously to what we did in Lemma \ref{lem: auxil 1}. Fix $i\in \{1,\ldots ,l-1\}$ and $\varepsilon> 0$ such that $2\varepsilon <\delta $. From Theorem 2 of \cite{DrF} it follows that there exists a measurable map $C:M\to (0,\infty)$ such that
$$\frac{1}{C(x)}e^{(\lambda _i-\varepsilon )n}\norm{u}\leq \norm{A^n_i(x)u}\leq C(x)e^{(\lambda _i+\varepsilon )n}\norm{u}$$
for every $n\in \mathbb{N}$ and $u\in E^{i,A}_{x}$. Equation \eqref{eq: comp norm lemma} from the proof of Lemma \ref{lem: auxil 1} tells us that
$$\frac{1}{C(x)}e^{-(\lambda _i+2\varepsilon )n}\norm{u}\leq \norm{A^{-n}_i(x)u}\leq C(x)e^{-(\lambda _i-2\varepsilon )n}\norm{u}$$
for every $n\in \mathbb{N}$. Combining this two equations we get that
\begin{displaymath}
\begin{split}
\norm{u}_{x,i}^2&=\sum _{n\in \mathbb{Z}}\norm{A^n_i(x)u}^2e^{-2\lambda_i n -2\delta \mid n\mid}\\
&\leq \sum _{n\in \mathbb{Z}}\left( C(x)e^{\lambda_i n + 2\varepsilon \mid n\mid}\norm{u}\right)^2e^{-2\lambda_i n -2\delta \mid n\mid}\\
&=C(x)^2\sum _{n\in \mathbb{Z}}e^{ (4\varepsilon -2\delta )\mid n\mid} \norm{u}^2.
\end{split}
\end{displaymath}
For $u\in E^{l,A}_x$, Theorem 2 of \cite{DrF} tells us that whenever $-\frac{1}{\delta} <\lambda _{l-1}$,
$$\norm{A^n(x)u}\leq C(x)e^{-\frac{1}{\delta}n}\norm{u}.$$
Thus, 
$$\norm{u}_{x,l}^2= \sum _{n\geq 0}\norm{A^n(x)u}^2e^{-2\frac{1}{\delta}n}\leq C(x)^2\sum _{n\geq 0}e^{-\frac{4}{\delta}n}\norm{u}^2.$$

Thus, taking $K=\max \{\sum _{n\in \mathbb{Z}}e^{ (4\varepsilon -2\delta )\mid n\mid}, \sum _{n\geq 0}e^{-\frac{4}{\delta}n} \}$ and writing $u\in \mathbb{R}^d$ as $u=u_1+\ldots +u_l $ where $u_i \in E^{i,A}_x$ for every $1\leq i\leq l$ we get that

$$\norm{u}_x^2=\sum_{i=1}^{l}\norm{u_i}_{x,i}^2\leq KC^2(x)\sum _{i=1}^{l}\norm{u_i}^2.$$

It remains to obtain an upper bound for $\lVert u_i\rVert$ in terms of $\lVert u\rVert$. This can be achieved by using the map $K$ given by Theorem 2 of \cite{DrF}. More precisely, let $K^1$ be the map given by \cite[Theorem 2]{DrF} applied for $i=1$ and sufficiently small $\epsilon >0$. We then have that
\begin{equation}\label{new1}
 \lVert u_1\rVert \le K^1(x)\lVert u\rVert \quad \text{and} \quad \lVert u_2+\ldots +u_{l}\rVert \le K^1(x)\lVert u\rVert .
\end{equation}
The first inequality in~\eqref{new1} gives a desired bound for $\lVert u_1\rVert$. In order to obtain the bound for $\lVert u_2\rVert$, we can apply again \cite[Theorem 2]{DrF}
but now for $i=2$ (and again for $\epsilon >0$ sufficiently small) to conclude that there exists $K^2$ such that
\begin{equation}\label{new2}
 \lVert u_2\rVert \le K^2(x)\lVert u_2+\ldots +u_{l}\rVert \quad \text{and} \quad \lVert u_3+\ldots +u_{l}\rVert \le K^2(x)\lVert u_2+\ldots +u_{l}\rVert.
\end{equation}
By combining the second inequality in~\eqref{new1} with the first inequality in~\eqref{new2}, we conclude that $\lVert u_2\rVert \le K^1(x)K^2(x)\lVert u\rVert$. By proceeding, one can establish desired bounds for all $\lVert u_j\rVert$, $j=1, \ldots, l$ and construct a function $K_\delta$ satisfying \eqref{ineq: K delta growth}. Indeed, this follows from the fact that $C(f^n(x))\leq C(x)e^{\varepsilon\mid n\mid}$ for every $n\in \mathbb{Z}$ and similarly for the maps $K^1,K^2,\ldots ,K^l$ completing the proof. 
\end{proof}

For any $N>0$, let $\Reg _{\delta ,N}$ be the set of regular points $x\in \Reg$ for which $K_{\delta}(x)\leq N$. Observe that $\mu(\Reg _{\delta, N})\to 1$ as $N\to +\infty$. Moreover, invoking Lusin's theorem we may assume without loss of generality that this set is compact and that the Lyapunov norm and the Oseledets splitting are continuous when restricted to it. 

As a final and simple remark about Lyapunov norms we observe that, in order to get a norm satisfying the properties given by the previous proposition, it is not necessary to use inner products. Indeed, for $x\in \Reg$, $\delta >0$ and $u=u_1+\ldots +u_l \in \R ^d$ where $u_i \in E^{i,A}_x$ for every $1\leq i\leq l$, defining 
\begin{displaymath}
\norm{u}_{x,\delta}=\sum _{i=1}^l\norm{u_i}_{x,\delta,i}
\end{displaymath}
where
\begin{equation*}
\norm{u_i}_{x,\delta,i}=\sum _{n\in \mathbb{Z}}\norm{A^{n}_i(x)u_i} e^{-\lambda _i n -\delta \mid n\mid}
\end{equation*}
for every $i $ for which $\lambda _i$ is finite and 
\begin{equation*}
\norm{u_l}_{x,\delta,l}= \sum _{n=0}^{+\infty}\norm{A^{n}(x)u_l}e^{\frac{1}{\delta}n}
\end{equation*}
in the case when $\lambda _l =-\infty$ gives rise to such a norm. In particular, this can be used to define Lyapunov norms for semi-invertible cocycles taking values on Banach spaces.

\section{Proof of the main result}

\subsection{Approximation of the largest Lyapunov exponent}

We start the proof of Theorem \ref{theo: main} with a key proposition which tells us that the largest Lyapunov exponent of $A$ may be approximated by Lyapunov exponents on periodic points. We retain all the notation introduced at the previous section.

\begin{proposition}\label{prop: apr largest exponent}
For every $\delta >0$ small enough, there exists a periodic point $p\in M$ such that
\begin{equation}\label{eq: aprox largest}
\mid \lambda _1 -\lambda_1(p)\mid <\delta.
\end{equation} 
More precisely, there exists $\delta _0>0$ such that for any $N>0$ and $\delta \in (0,\delta _0)$, there exist $n_0\in \N$ and $\rho >0$ such that, for every $n\geq n_0$, if $x, f^{n}(x)\in \Reg _{\delta ,N}$ are such that $d(x,f^n(x))<\rho $ and $p$ is a periodic point associated to $x$ by the Anosov Closing property then
\begin{equation*}
\mid \lambda _1 -\lambda_1(p)\mid <\delta.
\end{equation*} 
\end{proposition}

Let $C_1, \varepsilon _0, \theta >0$ be given by the Anosov Closing Property, $\rho \in (0,\varepsilon _0)$ and suppose $d(x,f^n(x))<\rho$. Thus, there exists a periodic point $p\in M$ of period $n$ such that
\begin{displaymath}
d(f^j(x),f^j(p))\leq C_1 e^{-\theta \min\lbrace j, n-j\rbrace}d(f^n(x),x)\leq C_1 \rho e^{-\theta \min\lbrace j, n-j\rbrace}
\end{displaymath}
for every $j=0,1,\ldots ,n$. We will prove that, as long as $n$ is sufficiently large, this periodic point satisfies the previous proposition. We split the proof into two lemmas. In the first one we give a lower bound for $\lambda _1(p)$ in terms of $\lambda _1$ while in the second one an upper bound is given. 

Consider $0<\delta _0 < \frac{1}{4}\min \{\theta \alpha , (\lambda _1 -\lambda _2) \}$ if $A$ has at least two different Lyapunov exponents and moreover $\lambda _2>-\infty$ and let $0< \delta _0 <\frac{1}{4}\theta \alpha$ otherwise. Assume also that $\delta _0$ is small enough so that $-\frac{1}{\delta _0}<\lambda _{l-1}$. Fix $N>0$ and $\delta \in (0,\delta _0)$ and suppose $x, f^n(x)\in \Reg _{\delta , N}$.

\begin{lemma}[Lower bound]
There exists $n_0\in \N$ such that, if $\rho \in (0,\varepsilon _0)$ is sufficiently small and $n\geq n_0$ then
\begin{equation}\label{eq: aprox largest }
\lambda_1(p)\geq \lambda _1 -\delta.
\end{equation} 
\end{lemma}

\begin{proof}
For each $1\leq j\leq n$, let us consider the splitting $\R ^d=E^{1,A}_{f^j(x)}\oplus F_{f^j(x)}$ where $F_{f^j(x)}=E^{2,A}_{f^j(x)}\oplus \ldots \oplus E^{l,A}_{f^j(x)}$ and write $u\in \R ^d$ as $u=u^j_E+u^j_F$ where $u^j_E\in E^{1,A}_{f^j(x)}$ and $u^j_F\in F_{f^j(x)}$. Then the \textit{cone} of radius $1-\gamma > 0$ around $E^{1,A}_{f^j(x)}$ is defined as  
\begin{displaymath}
C^j_{\gamma }=\left\{ u^j_E+u^j_F\in E^{1,A}_{f^j(x)}\oplus F_{f^j(x)}; \; \norm{u^j_F}_{f^j(x)}\leq (1-\gamma) \norm{u^j_E}_{f^j(x)}\right\}.
\end{displaymath}
To simplify notation we write $\norm{.}_j$ for the Lyapunov norm at the point $f^j(x)$.

We claim now that it is enough to prove that if $\rho $ is sufficiently small then there exists $\gamma \in (0,1)$ such that 
\begin{equation}\label{ineq: cone}
A(f^j(p))(C^j_{0})\subset C^{j+1}_{\gamma}
\end{equation}
and for every $u\in C^j_0$,
\begin{equation}\label{ineq: Lyap norm cone}
\norm{(A(f^j(p))u)^{j+1}_E}_{j+1}\geq e^{\lambda _1 -2\delta}\norm{u^j_E}_j.
\end{equation}
Indeed, since we are assuming that the Oseledets splitting and the Lyapunov norm are continuous on $\Reg _{\delta ,N}$, it follows that if $\rho $ is sufficiently small (and consequently $x$ and $f^n(x)$ are close) then $C^n_{\gamma}\subset C^0_0$ and thus by \eqref{ineq: cone}, $A^n(p)(C^0_{0})\subset C^{0}_{0}$. Consequently, for any $u\in C^0_0$ and $k\in \N$ we have that $A^{kn}(p)u\in C^0_0$. Therefore, given $u\in C^0_0$, invoking \eqref{ineq: Lyap norm cone} and the fact that the Lyapunov norms at $x$ and $f^n(x)$ are close whenever $\rho$ is small, 
\begin{displaymath}
\begin{split}
\norm{A^n(p)u}_n & \geq \norm{(A^n(p)u)^n_E}_n\geq  e^{n(\lambda _1 -2\delta)}\norm{u^0_E}_0 \\
&\geq \frac{1}{2} e^{n(\lambda _1 -2\delta)}\norm{u}_0\geq \frac{1}{4} e^{n(\lambda _1 -2\delta)}\norm{u}_n
\end{split}
\end{displaymath}
which applied $k$ times leads us to
\begin{displaymath}
\norm{A^{kn}(p)u}_n  \geq \frac{1}{4^k} e^{kn(\lambda _1 -2\delta)}\norm{u}_n.
\end{displaymath}
Consequently, 
\begin{displaymath}
\begin{split}
\lambda _1(p)&\geq \lim _{k\to \infty} \frac{1}{kn}\log \left(\norm{A^{kn}(p)u}_n \right) \geq \lim _{k\to \infty} \frac{1}{kn} \log \left( \frac{1}{4^k} e^{kn(\lambda _1 -2\delta)}\norm{u}_n \right)\\
&=\lambda _1-2\delta- \frac{\log 4}{n} +\frac{1}{n}\lim _{k\to \infty}\frac{1}{k}\log \left(\norm{u}_n\right) \geq \lambda_1 -3\delta
\end{split}
\end{displaymath}
as long as $n$ is large enough which proves our claim. So, the only thing left to do is to prove \eqref{ineq: cone} and \eqref{ineq: Lyap norm cone}. Assume initially that $\lambda _2>-\infty$.

Given $u\in C^j_0$ let us consider $v=A(f^j(x))u$. Then, it follows from \eqref{ineq: Lyapunov norm} that $\norm{v}_{j+1}\leq e^{\lambda _1 +\delta}\norm{u}_j$ and moreover that 
\begin{displaymath}
\norm{v^{j+1}_E}_{j+1}=\norm{A(f^j(x))u^j_E}_{j+1}\geq e^{\lambda _1 -\delta}\norm{u^j_E}_j
\end{displaymath}
and
\begin{equation}\label{eq: norm vjF}
\norm{v^{j+1}_F}_{j+1}=\norm{A(f^j(x))u^j_F}_{j+1}\leq e^{\lambda _2 +\delta}\norm{u^j_F}_j.
\end{equation}

Let $w=A(f^j(p))u$. What we want to do now is to compare the Lyapunov norms of $w$ and its projection on $E^{1,A}_{f^{j+1}(x)}$ and $F_{f^{j+1}(x)}$ with the respective norms of $v$. In order to do it, let us consider $B_j=A(f^j(p))-A(f^j(x))$. Consequently, $w=v+B_ju$ and thus $w^{j+1}_E=v^{j+1}_E+ (B_ju)^{j+1}_E$ and $w^{j+1}_F=v^{j+1}_F+ (B_ju)^{j+1}_F$. Moreover, for every $0\leq j\leq n$,
\begin{displaymath}
\begin{split}
\norm{B_j}&=\norm{A(f^j(p))-A(f^j(x))}\leq C_2 d(f^j(p),f^j(x))^{\alpha} \\
& \leq C_1 C_2 \rho ^{\alpha}e^{-\theta \alpha \min\{ j, n-j\}}.
\end{split}
\end{displaymath}
Therefore, invoking \eqref{ineq: norm x Lyapunov norm operator} and \eqref{ineq: norm x Lyapunov norm} it follows that
\begin{displaymath}
\norm{B_ju}_{j+1}  \leq \norm{B_j}_{f^{j+1}(x)\leftarrow f^{j+1}(x)}\norm{u}_{j+1}\leq K_{\delta}(f^{j+1}(x))^2\norm{B_j}\norm{u}.
\end{displaymath}
Then, using that $ K_{\delta}(f^{j+1}(x)) \leq N e^{\delta \min\{ j+1, n-j-1\}}$ which follows from \eqref{ineq: K delta growth} and the fact that $x$ and $f^n(x)$ are in $\Reg _{\delta,N}$ and that $\norm{u}_j\leq 2\norm{u_E^j}_j$ since $u\in C^j_0$, we get
\begin{displaymath}
\begin{split}
\norm{B_ju}_{j+1} &\leq  N^2 e^{2\delta\min\{ j+1, n-j-1\}} C_1 C_2 \rho ^{\alpha}e^{-\theta \alpha \min\{ j, n-j\}}\norm{u}_j\\
&\leq  C_1 C_2 N^2 \rho ^{\alpha} e^{2\delta\min\{ j+1, n-j-1\}} e^{-\theta \alpha \min\{ j, n-j\}}2\norm{u^j_E}_j\\
& \leq C \rho ^{\alpha}e^{(2\delta -\theta \alpha) \min\{ j, n-j\}}\norm{u^j_E}_j.
\end{split}
\end{displaymath}
Thus, since $2\delta -\theta \alpha <0$ we get that $\norm{B_ju}_{j+1}\leq \tilde{C}\rho ^{\alpha}\norm{u^j_E}_j$ for some $\tilde{C}>0$ independent of $n$ and $j$. Consequently,
\begin{displaymath}
\begin{split}
\norm{w^{j+1}_E}_{j+1} & \geq \norm{v^{j+1}_E}_{j+1} -\norm{(B_ju)^{j+1}_E}_{j+1}\\
&\geq e^{\lambda _1 -\delta}\norm{u^j_E}_j -\tilde{C}\rho ^{\alpha}\norm{u^j_E}_j \geq  e^{\lambda _1 -2\delta}\norm{u^j_E}_j
\end{split}
\end{displaymath}
whenever $\rho$ is small enough which is precisely inequality \eqref{ineq: Lyap norm cone}. To get \eqref{ineq: cone} we observe initially that, analogously to the previous inequality we can get
\begin{equation}\label{aux: lower 1}
\norm{w^{j+1}_E}_{j+1}  \leq  e^{\lambda _1 +\delta}\norm{u^j_E}_j +\tilde{C}\rho ^{\alpha}\norm{u^j_E}_j \leq  \hat{C}\norm{u^j_E}_j.
\end{equation}
On the other hand,
\begin{displaymath}
\norm{w^{j+1}_E}_{j+1}  \geq \norm{v^{j+1}_E}_{j+1}  -\norm{B_ju}_{j+1}
\end{displaymath}
and
\begin{displaymath}
\norm{w^{j+1}_F}_{j+1}  \leq \norm{v^{j+1}_F}_{j+1}  +\norm{B_ju}_{j+1}. 
\end{displaymath}
Therefore, combining this inequalities and using again that $u\in C^j_0$,
\begin{displaymath}
\begin{split}
\norm{w^{j+1}_E}_{j+1} - \norm{w^{j+1}_F}_{j+1} & \geq \norm{v^{j+1}_E}_{j+1} -\norm{v^{j+1}_F}_{j+1}  -2\norm{B_ju}_{j+1}\\
&\geq e^{\lambda _1-\delta}\norm{u^{j}_E}_{j} - e^{\lambda _2+\delta}\norm{u^{j}_F}_{j} -2\tilde{C}\rho ^{\alpha}\norm{u^j_E}_j\\
&\geq \left( e^{\lambda _1-\delta} - e^{\lambda _2+\delta} -2\tilde{C}\rho ^{\alpha}\right) \norm{u^j_E}_j.
\end{split}
\end{displaymath}
Taking $\rho$ small enough so that $e^{\lambda _1-\delta} - e^{\lambda _2+\delta} -2\tilde{C}\rho ^{\alpha}>0$ and applying \eqref{aux: lower 1} to the previous inequality we get that there exists $\gamma >0 $ such that 
$\norm{w^{j+1}_E}_{j+1} - \norm{w^{j+1}_F}_{j+1}\geq \gamma \norm{w^{j+1}_E}_{j+1}$ which implies that $w=A(f^j(p))u\in C^{j+1}_{\gamma}$ proving \eqref{ineq: cone} and consequently the lemma whenever $\lambda _2>-\infty$. The case when $\lambda _2=-\infty$ is analogous. The only difference is that inequality \eqref{eq: norm vjF} becomes
\begin{displaymath}
\norm{v^{j+1}_F}_{j+1}=\norm{A(f^j(x))u^j_F}_{j+1}\leq e^{-\frac{1}{\delta}}\norm{u^j_F}_j.
\end{displaymath}
\end{proof}

\begin{lemma}[Upper bound]
There exists a constant $c>0$ such that 
\begin{equation}
\norm{A^n(p)}_{f^n(x)\leftarrow x}\leq c e^{\rho ^\alpha} e^{(\lambda _1+\delta)n}
\end{equation}
and
\begin{equation}
\norm{A^n(p)}\leq c N e^{\rho ^{\alpha}} e^{(\lambda _1+\delta)n}.
\end{equation}
Consequently, if $\rho >0$ is sufficiently small and $n$ is large enough,
\begin{displaymath}
\lambda _1(p)\leq \lambda _1 +2\delta.
\end{displaymath}
\end{lemma}

\begin{proof} Let us consider $B_j=A(f^j(p))-A(f^j(x))$. As in the proof of the previous lemma we have that, for every $0\leq j\leq n$,
\begin{displaymath}
\begin{split}
\norm{B_j}& \leq C_1 C_2 \rho ^{\alpha}e^{-\theta \alpha \min\{ j, n-j\}} =C\rho ^{\alpha}e^{-\theta \alpha \min\{ j, n-j\}}
\end{split}
\end{displaymath}
and
\begin{equation}\label{ineq: norm of B j}
\begin{split}
\norm{B_j}_{f^j(x)\leftarrow f^j(x)} \leq K_{\delta}(f^j(x))\norm{B_j} \leq N C\rho ^{\alpha}e^{(\delta-\theta \alpha) \min\{ j, n-j\}}.
\end{split}
\end{equation}

Our objective now is to estimate $\norm{A^n(p)}_{f^n(x)\leftarrow x}$. We start observing that
\begin{displaymath}
\begin{split}
\norm{A^n(p)}_{f^n(x)\leftarrow x}&=\norm{A(f^{n-1}(p))\cdot \ldots \cdot A(p)}_{f^n(x)\leftarrow x}\\
&= \norm{(A(f^{n-1}(x))+B_{n-1})\cdot \ldots \cdot (A(x)+B_0)}_{f^n(x)\leftarrow x}\\
&\leq \norm{A(f^{n-1}(x))+B_{n-1}}_{f^{n}(x)\leftarrow f^{n-1}(x)}\cdot \ldots \cdot \norm{A(x)+B_{0}}_{f(x)\leftarrow x}
\end{split}
\end{displaymath}
and, for each $0\leq j< n$,
\begin{displaymath}
\norm{A(f^{j}(x))+B_{j}}_{f^{j+1}(x)\leftarrow f^{j}(x)}\leq \norm{A(f^{j}(x))}_{f^{j+1}(x)\leftarrow f^{j}(x)} +\norm{B_{j}}_{f^{j+1}(x)\leftarrow f^{j}(x)}.
\end{displaymath}
Thus, since from \eqref{ineq: norm}
\begin{displaymath}
\norm{A(f^{j}(x))}_{f^{j+1}(x)\leftarrow f^{j}(x)} \leq e^{(\lambda _1 +\delta )},
\end{displaymath}
invoking \eqref{ineq: norm of B j} we get that
\begin{displaymath}
\begin{split}
\norm{A(f^{j}(x))+B_{j}}_{f^{j+1}(x)\leftarrow f^{j}(x)}& \leq e^{(\lambda _1 +\delta )} +CN\rho ^{\alpha}e^{(\delta-\theta \alpha) \min\{ j, n-j\}}  \\
& = e^{(\lambda _1 +\delta )} (1+ CN \rho ^{\alpha} e^{-(\lambda _1 +\delta )}e^{(\delta-\theta \alpha) \min\{ j, n-j\}}).
\end{split}
\end{displaymath}
Making $\tilde{C}=CN e^{-(\lambda _1 +\delta )}$ and using the fact that $1+y\leq e^y$ for every $y\geq 0$ it follows that
\begin{displaymath}
\norm{A(f^{j}(x))+B_{j}}_{f^{j+1}(x)\leftarrow f^{j}(x)} \leq e^{(\lambda _1 +\delta )} \exp (\tilde{C}\rho ^{\alpha} e^{(\delta-\theta \alpha) \min\{ j, n-j\}}).
\end{displaymath}
Consequently,
\begin{displaymath}
\begin{split}
\norm{A^n(p)}_{f^n(x)\leftarrow x}& \leq \prod _{j=0} ^{n-1}  e^{(\lambda _1 +\delta )} \exp \left( \tilde{C} \rho ^{\alpha} e^{(\delta-\theta \alpha) \min\{ j, n-j\}}\right) \\
&= e^{(\lambda _1 +\delta )n}   \exp \left( \tilde{C} \rho ^{\alpha} \sum _{j=0}^{n-1} e^{( \delta-\theta \alpha) \min\{ j, n-j\})}\right).
\end{split}
\end{displaymath}
Now, using the fact that $\delta-\theta \alpha <0$ and making $c=\exp \left( \tilde{C} \sum _{j=0}^{\infty} 2e^{( \delta-\theta \alpha) j}\right)$ we get the first claim of the lemma. The second one follows from the previous one observing that $x\in \Reg _{\delta ,N}$ and $\norm{A^n(p)}\leq K_{\delta}(x)\norm{A^n(p)}_{f^n(x)\leftarrow x}$. To conclude the proof it only remains to observe that
$$\lambda _1(p)\leq \frac{1}{n}\log (\norm{A^n(p)})$$
which combined with the previous inequality implies 
\begin{displaymath}
\lambda _1(p)\leq \lambda _1 +\delta + \frac{1}{n} \log (cNe^{\rho ^{\alpha}}).
\end{displaymath}
Since $\rho$ may be taken arbitrary small the lemma follows.
\end{proof}

Proposition \ref{prop: apr largest exponent} now follows easily from these two lemmas.

\subsection{Approximation of the infinite Lyapunov exponent}

In this subsection we prove that even when the Lyapunov exponents are not finite we still can approximate it by Lyapunov exponents on periodic points. This will follow from the next general proposition. 

\begin{proposition} \label{prop: appr infinite Lyap}
Let $f:M\to M$ be a continuous map defined on the compact metric space $(M,d)$, $\mu$ an ergodic $f$-invariant probability measure and $B:M\to M(d,\R)$ a continuous map such that $\lambda _1(B,\mu)=-\infty$. If $\{\mu _j \}_{j\in \N}$ is sequence of ergodic $f$-invariant probability measures converging in the weak$^{\ast}$ topology to $\mu$, then
\begin{displaymath}
	\limsup _{j\to \infty} \lambda _1(B,\mu _j) =-\infty.
\end{displaymath}
\end{proposition}

\begin{proof}
For each $n\in\N$, let $\varphi _n :M\to [-\infty ,+\infty )$ be the map given by
\begin{displaymath}
	\varphi _n (x)=\dfrac{1}{n}\log \norm{B^n(x)}.
\end{displaymath}
We start observing that, since $\mu$ is ergodic, $\varphi _1^+ (x)=\max \{0,\varphi _1(x)\} \in \mbox{L}^1(\mu)$ and $\{n\varphi _n\}_{n\in \N}$ is a subadditive sequence, it follows by Kingman's Subadditive Theorem (see for instance \cite{LLE}) that
\begin{displaymath}
	\lambda (B,\mu)=\lim _{n\to \infty} \varphi _n(x)=\inf _n \int \varphi _n (x)d\mu
\end{displaymath}
for $\mu$ almost every $x\in M$.
Analogously, since for each $j\in \N$ the measure $\mu _j$ is ergodic and $\varphi _1^+ (x) \in \mbox{L}^1(\mu _j)$, we have
\begin{displaymath}
	\lambda (B,\mu _j)=\inf _n \int \varphi _n (x)d\mu _j.
\end{displaymath}
Therefore, in order to complete the proof it is enough to show that
\begin{equation*}
\limsup _{j\to +\infty} \inf _n \int \varphi _n (x)d\mu _j=-\infty.
\end{equation*}

Given $m\in \N$, let us consider $\varphi _{n,m}:M\to (-\infty, +\infty)$ given by
\begin{displaymath}
	\varphi _{n,m}(x)=\left\{\begin{array}{cc}
		\varphi _n(x) & \mbox{if} \quad \varphi _n(x)\geq -m\\
		-m & \mbox{if}\quad \varphi _n(x)< -m.
	\end{array}
	\right.
\end{displaymath}
It follows easily from the definition and from the properties of $\varphi _n$ that, for every $m,n\in \N$, $\varphi _{n,m}:M\to (-\infty, +\infty)$ is a continuous function. Moreover, 
\begin{displaymath}
\varphi _{n,m}(x)\geq \varphi _{n,m+1}(x) \quad \mbox{and} \quad \varphi _{n,m}(x)\xrightarrow{m\to +\infty} \varphi _{n}(x)
\end{displaymath}
for every $m,n\in \N$ and $x\in M$. 

By the Monotone Convergence Theorem we get
\begin{displaymath}
\int 	\varphi _n d\mu = \lim _{m\to \infty} \int \varphi _{n,m} d\mu =\inf _{m} \int \varphi _{n,m} d\mu.
\end{displaymath}
On the other hand, since $\mu _j \xrightarrow{w^{\ast}}\mu$ as $j$ goes to infinite and $\varphi _{n,m}$ is continuous, we have
\begin{displaymath}
	\int \varphi _{n,m}d\mu =\lim _{j\to \infty} \int \varphi _{n,m}d\mu _j \quad \forall m\in \N.
\end{displaymath}
Thus,
\begin{equation*}
\inf _n \int \varphi _n d\mu =\inf _n \left\{\inf _m \left[ \lim _{j\to \infty} \int \varphi _{n,m} d\mu _j \right]\right\}.
\end{equation*}
Now, for each $j$ and $m$ in $\N$ we have $\int \varphi _{n,m} d\mu _j \geq \inf _m\left\{ \int \varphi _{n,m} d\mu _j\right\}$ and thus
\begin{displaymath}
	\lim _{j\to \infty}\int \varphi _{n,m} d\mu _j \geq \limsup _{j\to \infty} \inf _m\left\{ \int \varphi _{n,m} d\mu _j\right\}
\end{displaymath}
for every $m\in \N$. Consequently,
\begin{equation}
\inf _m \left\{ \lim _{j\to \infty} \int \varphi _{n,m} d\mu _j \right\} \geq  \limsup _{j\to \infty} \left\{ \inf _m \int \varphi _{n,m} d\mu _j \right\}.
\label{eq: infinite LE estimative 1}
\end{equation}
Moreover, once again by the Monotone Convergence Theorem, 
\begin{equation*}
 \limsup _{j\to \infty} \left\{ \inf _m \int \varphi _{n,m} d\mu _j \right\}= \limsup _{j\to \infty}  \int \varphi _n d\mu _j.
\end{equation*}
Observing then that, as in \eqref{eq: infinite LE estimative 1},
\begin{equation*}
\inf _n \left\{ \limsup _{j\to \infty}  \int \varphi _n d\mu _j\right\} \geq \limsup _{j\to \infty} \inf _n \int \varphi _n d\mu _j
\end{equation*}
it follows that 
\begin{displaymath}
	-\infty =\lambda (B,\mu)=\inf _n \int \varphi _n d\mu \geq \limsup _{j\to \infty} \inf _n \int \varphi _n d\mu _j
\end{displaymath}
as we want.

\end{proof}

\subsection{Conclusion of the proof}
To complete the proof of our main result the idea is to apply Propositions \ref{prop: apr largest exponent} and \ref{prop: appr infinite Lyap} to the cocycle induced by $A$ on a suitable exterior power, which by now is a very standard trick.

We may assume without loss of generality that $\mu$ is non-atomic. Otherwise the theorem is trivial. Moreover, we assume $\lambda _l=-\infty$. In the case when $\lambda _l>-\infty$ we only need the first part of our argument.
Recall that $\gamma _1\geq \gamma _2\geq \ldots \geq \gamma _d$ are the Lyapunov exponents of $A$ with respect to $\mu$ counted with multiplicities and, for every $i \in \{1,\ldots , d\}$, let $\Lambda ^i(\mathbb{R}^d)$ be the $i$th exterior power of $\mathbb{R}^d$ which is the space of alternate $i$-linear forms on the dual $(\mathbb{R}^d)^*$ and $\Lambda ^iA(x):\Lambda ^i(\mathbb{R}^d) \to \Lambda ^i(\mathbb{R}^d)$ the cocycle induced by $A(x)$ on the $i$th exterior power. A very well known fact about this cocycle (see for instance \cite{LLE}) is that its Lyapunov exponents are  
\begin{equation*} \label{Lyapunov exterior power}
\{\gamma _{j_1} +\ldots +\gamma _{j_i}; \; 1\leq j_1<\ldots < j_i\leq d\}.
\end{equation*}
In particular, its largest Lyapunov exponent is given by $\gamma_1+\gamma _2+\ldots+\gamma _i$.

Let $N$ be large enough so that the intersection $G$ of the sets $\Reg _{\delta, N}$ associated to all the cocycles $\Lambda ^iA$ for $i=1,\ldots , \text{dim}(E^{1,A}_x\oplus \ldots \oplus E^{l-1,A}_x)$ have positive measure, that is, $\mu (G)>0$. Let 
\begin{displaymath}
	B(\mu)=\left\{ x\in M; \; \dfrac{1}{n}\sum _{i=0}^{n-1}\delta _{f^i(x)}\xrightarrow{n\to \infty} \mu \quad \mbox{in the weak$^{\ast}$ topology} \right\}
 \end{displaymath}
be the \emph{basin} of $\mu$. Since $\mu$ is ergodic, $B(\mu)$ has full measure. Let $x\in B(\mu)\cap G$ be so that $\mu (B(x,\frac{1}{k})\cap G)>0$ for every $k\in \mathbb{N}$ where $B(x,\frac{1}{k})$ is the ball of radius $\frac{1}{k}$ centered at $x$. By Poincar\'e's Recurrence Theorem there exists a sequence $(n_k)_{k\in \mathbb{N}}$ of positive integers so that $n_k\to +\infty$ and $f^{n_k}(x)\in B(x,\frac{1}{k})\cap G$ for each $k\in \mathbb{N}$. By the Anosov Closing property it follows that, for each $k$ sufficiently large, there exists a periodic point $p_k$ of period $n_k$ so that
\begin{equation}\label{eq: Anosov closing 2}
d(f^j(x),f^j(p_k))\leq C_1 e^{-\theta \min\lbrace j, n_k-j\rbrace}d(f^{n _k}(x),x)\leq \frac{C_1}{k} e^{-\theta \min\lbrace j, n_k-j\rbrace}
\end{equation}
for every $j=0,1,\ldots , n_k$. It follows then by Proposition \ref{prop: apr largest exponent} applied to the cocycles $\Lambda ^iA$ for $i=1,\ldots , \text{dim}(E^{1,A}_x\oplus \ldots \oplus E^{l-1,A}_x)$ that for every $\delta>0$ small, there exists $k_{\delta}\in \mathbb{N}$ so that for any $k\geq k_{\delta}$,
$$\mid (\gamma _1+\gamma_2+\ldots+\gamma_i)-(\gamma _1 (p_k)+\gamma_2(p_k)+\ldots+\gamma_i(p_k))\mid <\delta$$
for every $i=1,\ldots , \text{dim}(E^{1,A}_x\oplus \ldots \oplus E^{l-1,A}_x)$ where $\{\gamma _j(p_k)\}_{j=1}^d$ are the Lyapunov exponents of $A$ at the periodic point $p_k$ counted with multiplicities. In particular,
\begin{equation}\label{eq: appr finite Lyap}
 \lim _{k\to +\infty} \gamma _i(p_k)=\gamma _i
\end{equation}
for every $i=1,\ldots , \text{dim}(E^{1,A}_x\oplus \ldots \oplus E^{l-1,A}_x)$. To conclude the proof of our main result it remains to observe that \eqref{eq: appr finite Lyap} also holds for $i=\text{dim}(E^{1,A}_x\oplus \ldots \oplus E^{l-1,A}_x)+1,\ldots ,d$. But this follows easily from Proposition \ref{prop: appr infinite Lyap} applied to $\Lambda ^iA$ observing that $\lambda_1 (\Lambda ^iA,\mu)=-\infty$ for every $i>\text{dim}(E^{1,A}_x\oplus \ldots \oplus E^{l-1,A}_x)$ and that the sequence $(\mu _{p_k})_{k\in \mathbb{N}}$ of ergodic \emph{periodic measures} given by 
$$\mu _{p_k} =\dfrac{1}{n_k}\sum _{j=0}^{n_k-1}\delta _{f^j(p_k)}$$
converges to $\mu$ in the weak$^*$ topology which follows from the fact that $x\in B(\mu)$ and \eqref{eq: Anosov closing 2}. Consequently,
\begin{equation*}
 \lim _{k\to +\infty} \gamma _i(p_k)=\gamma _i
\end{equation*}
for every $i=1,\ldots , d$ completing the proof of Theorem \ref{theo: main}.



\begin{thebibliography}{99}



\bibitem{Bac15}
\newblock L.~Backes, 
\newblock Rigidity of fiber bunched cocycles, 
\newblock \emph{Bulletin of the Brazilian Mathematical Society}, \textbf{46} (2015), 163--179.



\bibitem{BK16}
\newblock L.~Backes and A.~Kocsard,
\newblock Cohomology of dominated diffeomorphism-valued cocycles over hyperbolic systems,
\newblock \emph{Ergodic Theory and Dynamical Systems}, \textbf{36} (2016), 1703--1722.


\bibitem{BP07}
\newblock L.~Barreira and Ya.~Pesin, 
\newblock \emph{Nonuniform hyperbolicity: dynamics of systems with nonzero {L}yapunov exponents}, 
\newblock Cambridge University Press, 2007.


\bibitem{Bow75} 
\newblock R.~Bowen, 
\newblock \emph{Equilibrium states and the ergodic theory of Anosov diffeomorphisms},
\newblock Lecture Notes in Mathematics 470. Springer-Verlag, 1975.

\bibitem{Cao03}
\newblock Y.~Cao, 
\newblock Non-zero Lyapunov exponents and uniform hyperbolicity, 
\newblock \emph{Nonlinearity}, \textbf{16} (2003), 1473--1479.


\bibitem{Dai10}
\newblock X.~Dai, 
\newblock On the approximation of Lyapunov exponents and a question suggested by Anatole Katok,
\newblock \emph{Nonlinearity}, \textbf{23} (2010), 513--528.


\bibitem{dLW10}
\newblock R.~de la Llave and A. Windsor, 
\newblock Livšic theorems for non-commutative groups including diffeomorphism groups and results on the existence of conformal structur.es for Anosov systems,
\newblock \emph{Ergodic Theory and Dynamical Systems}, \textbf{30} (2010), 1055--1100.


\bibitem{DrF}
\newblock D.~Dragi\v{c}evi\'{c} and G.~Froyland, 
\newblock H\"older continuity of {O}seledets splittings for semi-invertible operator cocycles, 
\newblock \emph{Ergodic Theory and Dynamical Systems}, to appear. 

\bibitem{FLQ10}
\newblock G.~Froyland, S.~LLoyd, and A.~Quas, 
\newblock Coherent structures and isolated spectrum for {P}erron--{F}robenius cocycles,
\newblock \emph{Ergodic Theory and Dynamical Systems}, \textbf{30} (2010), 729--756.

\bibitem{Kal11}
\newblock B.~Kalinin,
\newblock Liv\v{s}ic theorem for matrix cocycles,
\newblock \emph{Annals of Mathematics}, \textbf{173} (2011),1025--1042.

\bibitem{KS}
\newblock B.~Kalinin and V.~Sadovskaya, 
\newblock Periodic approximation of Lyapunov exponents for Banach cocycles, 
\newblock Preprint \url{https://arxiv.org/abs/1608.05757}. 

\bibitem{KH95} 
\newblock A.~Katok and B.~Hasselblatt, 
\newblock \emph{Introduction to the modern theory of dynamical systems},
\newblock Cambridge University Press, London-New York, 1995.

\bibitem{KP16}
\newblock A.~Kocsard and R.~Potrie,
\newblock Livi\v{s}ic theorem for low-dimensional diffeomorphism cocycles,
\newblock \emph{Commentarii Mathematici Helvetici}, \textbf{91} (2016), 39--64.

\bibitem{Liv71} 
\newblock A.~Liv\v{s}ic. 
\newblock Homology properties of Y-systems,
\newblock \emph{Math. Zametki}, \textbf{10} (1971), 758--763.

\bibitem{Liv72}
\newblock A.~Liv\v{s}ic, 
\newblock Cohomology of dynamical systems,
\newblock \emph{Math. USSR Izvestija}, \textbf{6} (1972), 1278--1301.

\bibitem{Ose68}
\newblock V.~Oseledets, 
\newblock A multiplicative ergodic theorem: {L}yapunov characteristic numbers for dynamical systems,
\newblock \emph{Trans. Moscow Math. Soc.} \textbf{19} (1968), 197--231.


\bibitem{Sa15}
\newblock V.~Sadovskaya, 
\newblock Cohomology of fiber bunched cocycles over hyperbolic systems,
\newblock \emph{Ergodic Theory and Dynamical Systems}, \textbf{35} (2015), 2669--2688.

\bibitem{LLE}
\newblock M.~Viana, 
\newblock \emph{Lectures on {L}yapunov {E}xponents}, 
\newblock Cambridge University Press, 2014.

\bibitem{WS10}
\newblock Z.~Wang and W.~Sun, 
\newblock Lyapunov exponents of hyperbolic measures and hyperbolic periodic orbits, 
\newblock \emph{Trans. Amer. Math. Soc.}, \textbf{362} (2010), 4267--4282.




\end{thebibliography}
\end{document}